\documentclass[11pt]{article}

\usepackage[utf8]{inputenc}
\usepackage[T1]{fontenc}
\usepackage{amsmath, amssymb}
\usepackage{graphicx}
\usepackage{hyperref}
\usepackage{geometry,xcolor}

\usepackage{amsthm}

\usepackage{amsmath, amssymb, amsthm}

\theoremstyle{plain}
\newtheorem{theorem}{Theorem}[section]
\newtheorem{lemma}{Lemma}[section]

\theoremstyle{definition}
\newtheorem{definition}{Definition}[section]

\theoremstyle{remark}

\numberwithin{equation}{section}
\begin{document}

\title{Green’s function for an initial–boundary value problem with the regularized Prabhakar fractional derivative}

\author{
Erkinjon Karimov\\
Department of Mathematics: Analysis, Logics and Discrete Mathematics, Ghent University,\\
Krijgslaan 297, S8, Ghent 9000, Belgium\\
\texttt{erkinjon.karimov@ugent.be}
\and
Doniyor Usmonov\\
Department of Mathematical Analysis and Differential Equations, Fergana State University,\\
Murabbiylar Str. 19, 150100 Fergana, Uzbekistan\\
\texttt{dusmonov909@gmail.com}
\and
Maftuna Mirzaeva\\
Department of Mathematical Analysis and Differential Equations, Fergana State University,\\
Murabbiylar Str. 19, 150100 Fergana, Uzbekistan\\
\texttt{maftunamirzayeva2009@gmail.com}
}

\date{\today}

\maketitle

\begin{abstract}
This paper studies the first initial–boundary value problem for a sub-diffusion equation involving the regularized Prabhakar fractional derivative. By the superposition method, the problem is reduced to two auxiliary problems. An explicit solution representation is obtained in terms of a Green’s function, which is expressed via a bivariate Mittag-Leffler-type function. Existence and uniqueness of a regular solution are proved.

\end{abstract}

\vspace{0.3cm}


\section{Introduction}
Fractional differential equations extend classical differential equations to non-integer orders, allowing the modeling of systems with memory and hereditary properties \cite{Podlubny}. They find applications across physics -- such as anomalous diffusion and viscoelasticity \cite{Suzuki} -- engineering (control systems and signal processing), biology (population dynamics), medicine (pharmacokinetics), and finance \cite{Kilbas}. By capturing long-range dependencies, fractional models often provide more accurate descriptions than their integer-order counterparts \cite{El-Sayed}.

Green’s functions offer integral representations for boundary value problems, while well-establi\-shed for classical equations, recent research has extended their use to fractional cases. For example, Green’s functions for time-fractional diffusion-wave equations have been developed using potential methods \cite{Pskhu}, and explicit forms for fractional partial differential equations have been obtained by Laplace and Fourier transform techniques \cite{Sarah}, \cite{Odibat}.

The Prabhakar fractional operator, based on the Mittag-Leffler function of three-parameters, generalizes both both Riemann-Liouville and Caputo derivatives \cite{Giusti}, \cite{Garra}. It provides a framework for modeling memory effects and has applications in physics, renewal processes, and numerical analysis. Recent advancements have seen the application of Mikusiński’s operational calculus to Prabhakar-type operators, providing a robust algebraic framework for finding explicit series solutions to linear FDEs \cite{Noosha}. The study of Prabhakar-type operators has also been extended to bivariate structures and general analytic kernels, allowing for a more nuanced modeling of memory effects across multiple variables \cite{Isah}. Furthermore, the structural properties and inequalities of Prabhakar-type calculus, including the introduction of the Hilfer-Prabhakar and $(k,s)$-generalized derivatives, have been extensively explored to broaden their applicability in mathematical physics \cite{Tom1}, \cite{Tom2}.

Recent studies have explored the existence and uniqueness of solutions for boundary value problems and diffusion-wave equations involving Prabhakar derivatives, highlighting their increasing significance in fractional calculus \cite{Al‑Refai} (maximum principles), \cite{Asjad} (application in transport phenomena), \cite{Magar} (generalization in a sense of $\Psi$-fractional calculus), \cite{Karimov} (initial-boundary problem for sub-diffusion equation in a bounded domain), \cite{jonibek} (initial-boundary-value problem in star graph), \cite{Turdiev 1} (initial-boundary problem for fractional generalization of hyperbolic-type equation), \cite{Kerbal} (Katugampola-Prabhakar generalization), \cite{Mamanazarov} (inverse problems for time and space degenerate equation), \cite{Karimov2} (inverse problem for sub-diffusion equation with the Hilfer-Prabhakar FD), \cite{Usmonov} (Cauchy problem for sub-diffusion equation with the Prabhakar FD), \cite{WMT} (distributed order diffusion with Hilfer-Prabhakar fractional derivative).

In \cite{Karimov3}, a first boundary value problem with the Prabhakar fractional derivative was studied, where the Green’s function was constructed using a Volterra-type integral equation, leading to an explicit solution representation and proofs of existence and uniqueness. Based on these results, the present work investigates the first boundary value problem for a sub-diffusion equation involving the regularized fractional derivative of Prabhakar, derives an explicit solution and the corresponding Green’s function expressed in terms of a bivariate Mittag–Leffler-type function, and provides a basis for further effective developments in this direction.

\section{Main result} \setcounter{equation}{0}\label{sec:2}

Let $\Omega$ be a rectangular domain: $\Omega=\left\{ \left( t,x \right):\,\,0<t<T,\,\,0<x<a \right\},$ $0<a,T<\infty .$ 
We formulate the first initial-boundary value problem for the following sub-diffusion equation:
\begin{equation}\label{eq2.1}
    \mathop{Lu}={}^{PC}D_{0t}^{\alpha ,\,\beta ,\,\gamma ,\,\delta }u\left( t,x \right)-{{u}_{xx}}\left( t,x \right)=f\left( t,x \right).
\end{equation}
Here $f\left( t,x \right)$ is a given function, $\alpha,\beta,\gamma,\delta$ are given real numbers such that $\alpha>0,$ $0<\beta \le 1$,
$${ }^{PC} D_{0 t}^{\alpha, \beta, \gamma, \delta} g(t)={}^{P}I_{0 t}^{\alpha, 1-\beta,-\gamma, \delta} {\frac{d^{}}{d t}}g(t),\,\, t>0
$$ is the Prabhakar fractional derivative of Caputo type (regularized Prabhakar derivative) \cite{Garra},
$${ }^P I_{0 t}^{\alpha, \beta, \gamma, \delta} g(t)=\int\limits_0^t(t-s)^{\beta-1} E_{\alpha, \beta}^\gamma\left[\mathcal{\delta}(t-s)^\alpha\right] g(s)ds$$ is the Prabhakar fractional integral operator, whereas
$$
E_{\alpha, \beta}^\gamma[z]=\sum\limits_{k=0}^{+\infty} \frac{(\gamma)_k z^k}{\Gamma(\alpha k+\beta) k!}
$$
represents a generalized Mittag-Leffler (Prabhakar) function \cite{Prabhakar}. 

\begin{definition}
    A regular solution of the equation \eqref{eq2.1} in the domain $\Omega$ is called a function $u\left( t,x \right)$ with the regularity 
    $$u\left( t,x \right)\in C\left( \overline{\Omega} \right),\,\,{{u}_{xx}}\left( t,x \right),\, {}^{PC}D_{0t}^{\alpha ,\,\beta ,\,\gamma ,\,\delta }u\left( t,x \right)\in C\left( \Omega \right)$$ that satisfies the equation \eqref{eq2.1} at all points $\left( t,x \right)\in \Omega.$
\end{definition}

\textbf{Problem.} Find a regular solution $u\left( t,x \right)$ of the equation \eqref{eq2.1} in the domain $\Omega$, satisfying the following boundary and initial conditions:
\begin{equation}\label{eq2.2}
    u\left( t,0 \right)={{\varphi }_{0}}\left( t \right),\,\,\,   u\left( t,a \right)={{\varphi }_{1}}\left( t \right), \,\,\,   0\le t\le T,	
\end{equation}
\begin{equation}\label{eq2.3}
    u\left( 0,x \right)=\tau \left( x \right),\,\,\,  0\le x\le a,			 
\end{equation}
where ${{\varphi }_{0}}\left( t \right),$ ${{\varphi }_{1}}\left( t \right),$ $\tau \left( x \right)$ are given smooth functions.

\begin{theorem} 
 Let\, ${{\varphi }_{0}}\left( t \right),$ ${{\varphi }_{1}}\left( t \right)\in C\left[ 0;T \right],$ $\tau \left( x \right)\in C\left[ 0;a \right],$ ${{t}^{1-\beta }}f\left( t,x \right)\in C\left( \overline{\Omega} \right)$ and $f\left( t,x \right)$ also satisfies the H{\"o}lder condition with respect to $x$, and the conditions
 $${{\varphi }_{0}}\left( 0 \right)=\tau \left( 0 \right),\,\,\, {{\varphi }_{1}}\left( 0 \right)=\tau \left( a \right).$$
admits a unique regular solution given by 
$$u\left( t,x \right)=\int\limits_{0}^{t}{{{\varphi }_{0}}\left( \eta  \right){{G}_{\xi}}\left( t,x,\eta ,0 \right)d\eta }-\int\limits_{0}^{t}{{{\varphi }_{1}}\left( \eta  \right){{G}_{\xi}}\left( t,x,\eta ,a \right)d\eta }+$$
\begin{equation}\label{eq2.4}
    +\,\int\limits_{0}^{a}{\tau \left( \xi \right)\widetilde G(t,x,0,\xi)d\xi}+\int\limits_{0}^{t}{\int\limits_{0}^{a}{f\left( \eta ,\xi \right) G\left( t,x,\eta ,\xi \right)d\xi d\eta }}.
\end{equation}
\end{theorem}
Here 
$$G\left( t,x,\eta ,\xi \right)=\frac{{{\left( t-\eta  \right)}^{{{\beta }_{1}}-1}}}{2}\sum\limits_{n=-\infty }^{\infty }{\left[ {{E}_{12}}\left( \left. \begin{matrix}
   -{{\gamma }_{1}},1,{{\gamma }_{1}};\,\,\,\,\,\,\,\,\,\,\,\,\,\,\,\,\,\,\,\,\,\,\,\,\,\,\,  \\
   -{{\beta }_{1}},\alpha ,{{\beta }_{1}};-{{\gamma }_{1}},{{\gamma }_{1}};1,1;1,1  \\
\end{matrix} \right|\begin{matrix}
   -\left| x-\xi+2an \right|{{\left( t-\eta  \right)}^{-{{\beta }_{1}}}}  \\
   \delta {{\left( t-\eta  \right)}^{\alpha }}  \\
\end{matrix} \right) \right.}-$$
\begin{equation}\label{eq2.5}
    \left. -{{E}_{12}}\left( \left. \begin{matrix}
   -{{\gamma }_{1}},1,{{\gamma }_{1}};\,\,\,\,\,\,\,\,\,\,\,\,\,\,\,\,\,\,\,\,\,\,\,\,\,\,\,  \\
   -{{\beta }_{1}},\alpha ,{{\beta }_{1}};-{{\gamma }_{1}},{{\gamma }_{1}};1,1;1,1  \\
\end{matrix} \right|\begin{matrix}
   -\left| x+\xi+2an \right|{{\left( t-\eta  \right)}^{-{{\beta }_{1}}}}  \\
   \delta {{\left( t-\eta  \right)}^{\alpha }}  \\
\end{matrix} \right) \right],
\end{equation}
\begin{equation}\label{eq2.6}
    \widetilde G(t,x,0,\xi)=\int\limits_{0}^{t}\eta^{-\beta}E_{\alpha, 1-\beta}^{-\gamma}[\delta \eta^\alpha]{G\left( t,x,\eta ,\xi \right) d\eta }.
\end{equation}
$${{E}_{12}}\left( \left. \begin{matrix}
   {{\alpha }_{1}},{{\widetilde\beta }_{1}},{{\delta }_{1}};\,\,\,\,\,\,\,\,\,\,\,\,\,\,\,\,\,\,\,\,\,\,\,\,\,\,\,\,\,\,\,\,\,\,\,\,\,\,\,\,\,  \\
   {{\alpha }_{2}},{{\widetilde\beta }_{2}},{{\delta }_{2}};{{\alpha }_{3}},{{\delta }_{3}};{{\alpha }_{4}},{{\delta }_{4}};{{\widetilde\beta }_{3}},{{\delta }_{5}}  \\
\end{matrix} \right|\begin{matrix}
   x  \\
   y  \\
\end{matrix} \right)=$$
$$=\sum\limits_{n=0}^{+\infty }{\sum\limits_{m=0}^{+\infty }{\frac{\Gamma \left( {{\alpha }_{1}}n+{{\widetilde\beta }_{1}}m+{{\delta }_{1}} \right){{x}^{n}}{{y}^{m}}}{\Gamma \left( {{\alpha }_{2}}n+{{\widetilde\beta }_{2}}m+{{\delta }_{2}} \right)\Gamma \left( {{\alpha }_{3}}n+{{\delta }_{3}} \right)\Gamma \left( {{\alpha }_{4}}n+{{\delta }_{4}} \right)\Gamma \left( {{\widetilde\beta }_{3}}m+{{\delta }_{5}} \right)}}},$$
$$\left( x,y,{{\alpha }_{l}},{{\widetilde\beta }_{i}},{{\delta }_{j}}\in \mathbb{R};\min \left\{ {{\alpha }_{l}},{{\widetilde\beta }_{i}} \right\}>0;\left( l=\left\{ 1,...,4 \right\},i=\left\{ 1,2,3 \right\}, j=\left\{ 1,...,5 \right\} \right) \right),$$
in which the double series converges for $x,y\in \mathbb{R}$, if ${{\Delta }_{1}}>0$, and ${{\Delta }_{2}}>0$, whereas ${{\Delta }_{1}}={{\alpha }_{2}}+{{\alpha }_{3}}+{{\alpha }_{4}}-{{\alpha }_{1}}$, ${{\Delta }_{2}}={{\widetilde\beta }_{2}}+{{\widetilde\beta }_{3}}-{{\widetilde\beta }_{1}}$, ${{\beta }_{1}}=\frac{\beta }{2},\,\,\,{{\gamma }_{1}}=\frac{\gamma }{2}$ \cite{Turdiev 2}.

\textit{\bf Proof:}

{\bf Step 1: Decomposition.} 
We search the solution of problem \eqref{eq2.1}-\eqref{eq2.2}-\eqref{eq2.3} in the following form:
\begin{equation}\label{eq2.7}
    u(t,x)=y(t,x)+z(t,x).
\end{equation}
Here $y(t,x)$ is a solution of the nonhomogeneous equation
\begin{equation}\label{eq2.8}
    {}^{PC}D_{0t}^{\alpha ,\,\beta ,\,\gamma ,\,\delta }y\left( t,x \right)-{{y}_{xx}}\left( t,x \right)=f\left( t,x \right)
\end{equation}
that satisfies the nonhomogeneous boundary conditions
\begin{equation}\label{eq2.9}
    y\left( t,0 \right)={{\varphi }_{0}}\left( t \right),\,\,\,   y\left( t,a \right)={{\varphi }_{1}}\left( t \right), \,\,\,   0\le t\le T,	
\end{equation}
and the homogeneous initial condition
    \begin{equation}\label{eq2.10}
    y(0,x)=0,\,\,\,\,\, 0\le x \le a,
    \end{equation}
while $z(t,x)$ is a solution to the homogeneous equation 
\begin{equation}\label{eq2.11}
    {}^{PC}D_{0t}^{\alpha ,\,\beta ,\,\gamma ,\,\delta }z\left( t,x \right)-{{z}_{xx}}\left( t,x \right)=0
\end{equation}
that satisfies these homogeneous boundary conditions
\begin{equation}\label{eq2.12}
    z\left( t,0 \right)=0,\,\,\,   z\left( t,a \right)=0, \,\,\,   0\le t\le T	
\end{equation}
and the nonhomogeneous initial condition 
\begin{equation}\label{eq2.13}
    z\left( 0,x \right)=\tau \left( x \right),\,\,\,\,\,  0\le x\le a.	 
\end{equation}

{\bf Step 2: Solution of auxiliary problem.} First, we solve the problem \eqref{eq2.8}-\eqref{eq2.9}-\eqref{eq2.10}.
For this purpose, we use the following formula, which holds between the Prabhakar derivative of Riemann-Liouville type and the Prabhakar derivative of Caputo type for any function $g\in A{{C}^{m}}\left( a,b \right)$ \cite{Fernandez}:
\begin{equation}\label{2.14}
    {}^{PC}D_{at}^{\alpha ,\beta ,\gamma ,\delta }g\left( t \right)={}^{PRL}D_{at}^{\alpha ,\beta ,\gamma ,\delta }\left[ g\left( t \right)-\sum\limits_{j=0}^{m-1}{\frac{{{g}^{\left( j \right)}}\left( a \right)}{j!}{{\left( t-a \right)}^{j}}} \right],
\end{equation}
where 
$${ }^{PRL} D_{0 t}^{\alpha, \beta, \gamma, \delta} g(t)={\frac{d^{m}}{d t^m}}{}^{P}I_{0 t}^{\alpha, m-\beta,-\gamma, \delta} g(t),\,\,\, t>0$$
is the Prabhakar derivative of Riemann-Liouville type \cite{Prabhakar},
$\alpha ,\beta ,\delta ,\gamma \in \mathbb{C}$ with $\operatorname{Re}\left( \alpha  \right)>0,$ $\operatorname{Re}\left( \beta  \right)\ge 0$ and $m=\left[ \operatorname{Re}\left( \beta  \right) \right]+1.$

In our case, $m=\left[ \operatorname{Re}\left( {{\beta }} \right) \right]+1=1,$ $y( 0,x)=0.$ Hence, the following equality is true:
$${}^{PC}D_{0t}^{\alpha ,\,\beta ,\,\gamma ,\,\delta }y\left( t,x \right)={}^{PRL}D_{0t}^{\alpha ,\,\beta ,\,\gamma ,\,\delta }y\left( t,x \right).$$
Thus, the equation \eqref{eq2.8} can be written as follows:
\begin{equation}\label{eq2.15}
    {}^{PRL}D_{0t}^{\alpha ,\,\beta ,\,\gamma ,\,\delta }y\left( t,x \right)-{{y}_{xx}}\left( t,x \right)=f\left( t,x \right).
\end{equation}

Now let us present the following statement:
\begin{lemma}
    If $g\left( t,x \right)\in C\left( \overline{\Omega} \right),$ then 
     $\underset{t\to 0}{\mathop{\lim }}\,I_{0t}^{\alpha ,\,1-\beta,\,-\gamma,\,\delta }g\left( t,x \right)=0.$
\end{lemma}

\textit{\bf Proof.}  
Let us calculate the following expression:
$$\underset{t\to 0}{\mathop{\lim }}\,I_{0t}^{\alpha ,\,1-\beta,\,-\gamma,\,\delta }g\left( t,x \right)=\underset{t\to 0}{\mathop{\lim }}\,\int\limits_{0}^{t}{{{\left( t-z \right)}^{-\beta}}}E_{\alpha ,1-\beta}^{-\gamma}\left[ \delta {{\left( t-z \right)}^{\alpha }} \right]g\left( z,x \right)dz.$$
Since $g\left( t,x \right)\in C\left( \overline{\Omega} \right),$ it follows that the expression $g\left( t,x \right)$ is bounded. We define that
$$\underset{t\to 0}{\mathop{\lim }}\,\left|\int\limits_{0}^{t}{{{\left( t-z \right)}^{-\beta}}}E_{\alpha ,1-\beta}^{-\gamma}\left[ \delta {{\left( t-z \right)}^{\alpha }} \right]g\left( z,x \right)dz\right|\le $$
$$\le \left\| g\left( z,x \right) \right\|\underset{t\to 0}{\mathop{\lim }}\,\left|\int\limits_{0}^{t}{{{\left( t-z \right)}^{-\beta}}}E_{\alpha ,1-\beta}^{-\gamma}\left[ \delta {{\left( t-z \right)}^{\alpha }} \right]dz\right|=$$
$$=\left\| g\left( z,x \right) \right\|\underset{t\to 0}{\mathop{\lim }}\left|\sum\limits_{k=0}^{+\infty }{\frac{{{\left( -\gamma \right)}_{k}}{{\delta }^{k}}}{k!\Gamma \left( \alpha k-\beta+1 \right)}}\,\int\limits_{0}^{t}{{{\left( t-z \right)}^{\alpha k-\beta}}}dz\right|=$$
$$=\left\| g\left( z,x \right) \right\|\underset{t\to 0}{\mathop{\lim }}\left|\sum\limits_{k=0}^{+\infty }{\frac{{{\left( -\gamma \right)}_{k}}{{\delta }^{k}}}{k!\Gamma \left( \alpha k-\beta+2 \right)}}{{t}^{\alpha k-\beta+1}}\right|=\left\| g\left( z,x \right) \right\|\underset{t\to 0}{\mathop{\lim }}\,t^{1-\beta}E_{\alpha, 2-\beta}^{-\gamma}(\delta t^\alpha).$$
Since $0<\beta \le 1$, we obtain $\underset{t\to 0}{\mathop{\lim }}\,I_{0t}^{\alpha ,\,1-\beta,\,-\gamma,\,\delta }g\left( t,x \right)=0.$

Lemma 2.1 is proved.

{\bf Step 3: Green function construction.}
According to Lemma 2.1, we obtain the following initial condition for the equation \eqref{eq2.15}:
\begin{equation}\label{eq2.16}
    \underset{t\to 0}{\mathop{\lim }}\,I_{0t}^{\alpha ,\,1-\beta,\,-\gamma,\,\delta }y\left( t,x \right)=0.
\end{equation}
The solution of the problem \eqref{eq2.15}-\eqref{eq2.9}-\eqref{eq2.16} is expressed as follows (see \cite{Karimov3}):
\begin{equation}\label{eq2.17}
    y\left( t,x \right)=\int\limits_{0}^{t}{{{\varphi }_{0}}\left( \eta  \right){{G}_{\xi}}\left( t,x,\eta ,0 \right)d\eta }-\int\limits_{0}^{t}{{{\varphi }_{1}}\left( \eta  \right){{G}_{\xi}}\left( t,x,\eta ,a \right)d\eta }+\int\limits_{0}^{t}{\int\limits_{0}^{a}{f\left( \eta ,\xi \right)G\left( t,x,\eta ,\xi \right)d\xi d\eta }}.
\end{equation}

Now, we consider the problem \eqref{eq2.11}-\eqref{eq2.12}-\eqref{eq2.13}. Applying \eqref{2.14} for the function $z(t,x)$, we get
$${}^{PC}D_{0t}^{\alpha ,\,\beta ,\,\gamma ,\,\delta }z\left( t,x \right)={}^{PRL}D_{0t}^{\alpha ,\,\beta ,\,\gamma ,\,\delta }z\left( t,x \right)-\tau(x)t^{-\beta}E_{\alpha, 1-\beta}^{-\gamma}[\delta t^\alpha].$$
Using the last equality, we rewrite equation \eqref{eq2.11} and obtain nonhomogeneous equation 
\begin{equation}\label{eq2.18}
    {}^{PRL}D_{0t}^{\alpha ,\,\beta ,\,\gamma ,\,\delta }z\left( t,x \right)-{{z}_{xx}}\left( t,x \right)=\tau(x)t^{-\beta}E_{\alpha, 1-\beta}^{-\gamma}[\delta t^\alpha].
\end{equation}
Using Lemma 2.1, we determine the initial condition associated with the equation \eqref{eq2.18}:
\begin{equation}\label{eq2.19}
    \underset{t\to 0}{\mathop{\lim }}\,I_{0t}^{\alpha ,\,1-\beta,\,-\gamma,\,\delta }z\left( t,x \right)=0.
\end{equation}
Based on \cite{Karimov3}, the solution to problem \eqref{eq2.18}-\eqref{eq2.19}-\eqref{eq2.12} can be written as follows:
$$z(t,x)=\int\limits_{0}^{t}{\int\limits_{0}^{a}\tau(\xi)\eta^{-\beta}E_{\alpha, 1-\beta}^{-\gamma}[\delta \eta^\alpha]{G\left( t,x,\eta ,\xi \right)d\xi d\eta }}.$$
By introducing the substitution \eqref{eq2.6}, we can express $z(t,x)$ as follows:
\begin{equation}\label{eq2.20}
    z(t,x)=\int\limits_{0}^{a}\tau(\xi)\widetilde G(t,x,0,\xi)\,d\xi.
\end{equation}
Substituting the solutions \eqref{eq2.20} and \eqref{eq2.17} into  \eqref{eq2.7}, we obtain the solution \eqref{eq2.4}.

{\bf Step 4: Verification.} Now we prove that the function $u(t,x)$, represented by formula \eqref{eq2.4}, is indeed the solution we sought.

Let us introduce the following notation:
$$u_0(t,x)=u(t,x)-u_1(t,x),\,\,\,u_1(t,x)=\int\limits_0^t\int\limits_0^af(\eta,\xi)G(t,x,\eta,\xi)d\xi d\eta.$$

First, we show that $\mathop{Lu_0=0}.$ Let us begin by calculating the term involving the function $\varphi_0(\eta):$
$${}^{PC}D_{0t}^{\alpha ,\,\beta ,\,\gamma ,\,\delta }\left(\int\limits_{0}^{t}{{{\varphi }_{0}}\left( \eta  \right){{G}_{\xi}}\left( t,x,\eta ,0 \right)d\eta }\right)={}^{PRL}D_{0t}^{\alpha ,\,\beta ,\,\gamma ,\,\delta }\left(\int\limits_{0}^{t}{{{\varphi }_{0}}\left( \eta  \right){{G}_{\xi}}\left( t,x,\eta ,0 \right)d\eta }\right)=$$
$$={\frac{\partial^{}}{\partial t}}{}^{P}I_{0 t}^{\alpha, 1-\beta,-\gamma, \delta} \left(\int\limits_{0}^{t}{{{\varphi }_{0}}\left( \eta  \right){{G}_{\xi}}\left( t,x,\eta ,0 \right)d\eta }\right)=$$
$$=\frac{\partial}{\partial t}\int\limits_0^t(t-s)^{-\beta} E_{\alpha, 1-\beta}^{-\gamma}\left[\mathcal{\delta}(t-s)^\alpha\right]ds\int\limits_{0}^{s}{{{\varphi }_{0}}\left( \eta  \right){{G}_{\xi}}\left( s,x,\eta ,0 \right)d\eta }=$$
$$=\frac{\partial}{\partial t}\int\limits_0^t{{\varphi }_{0}}\left( \eta  \right)d\eta\int\limits_{\eta}^{t}(t-s)^{-\beta} E_{\alpha, 1-\beta}^{-\gamma}\left[\mathcal{\delta}(t-s)^\alpha\right]{{{G}_{\xi}}\left( s,x,\eta ,0 \right)ds, }$$
where \cite{Karimov3}
$${{G}_{\xi}}\left( t,x,\eta ,0 \right)=\sum_{n=-\infty}^{+\infty}sign(x+2na)\omega(t-\eta,\left|x+2na\right|),$$

$$\omega \left( t,x \right)=\sum\limits_{n=0}^{+\infty }{\frac{{{\left( -1 \right)}^{n}}{{x}^{n}}}{n!}{{t}^{-{{\beta }_{1}}n-1}}E_{\alpha ,-{{\beta }_{1}}n}^{-{{\gamma }_{1}}n}\left[ \delta {{t}^{\alpha }} \right]}={{t}^{-1}}{{E}_{12}}\left( \left. \begin{matrix}
   -{{\gamma }_{1}},1,0;\,\,\,\,\,\,\,\,\,\,\,\,\,\,\,\,\,\,\,\,\,\,\,\,\,\,\,  \\
   -{{\beta }_{1}},\alpha ,0;-{{\gamma }_{1}},0;1,1;1,1  \\
\end{matrix} \right|\begin{matrix}
   -x{{t}^{-{{\beta }_{1}}}}  \\
   \delta {{t}^{\alpha }}  \\
\end{matrix} \right).$$

After certain evaluations (see Appendix A1), we obtain
$${}^{PC}D_{0t}^{\alpha ,\,\beta ,\,\gamma ,\,\delta }\left(\int\limits_{0}^{t}{{{\varphi }_{0}}\left( \eta  \right){{G}_{\xi}}\left( t,x,\eta ,0 \right)d\eta }\right)=\int\limits_0^t{{\varphi }_{0}}\left( \eta  \right)\sum_{n=-\infty}^{+\infty}sign(x+2na)\times$$
\begin{equation}\label{eq2.23}
    \times\sum_{k=0}^{+\infty}\frac{(-1)^k\left|x+2na\right|^k(t-\eta)^{-\beta_1k-\beta-1}}{k!}E_{\alpha, -\beta_1k-\beta}^{-\gamma-\gamma_1k}\left[\delta(t-\eta)^{\alpha }\right]d\eta.
\end{equation}

Now we compute the second-order partial derivative with respect to 
$x$ of the term involving $\varphi_0(\eta):$
\begin{equation}\label{eq2.24}
    \frac{\partial^2}{\partial x^2}\left(\int\limits_{0}^{t}{{{\varphi }_{0}}\left( \eta  \right){{G}_{\xi}}\left( t,x,\eta ,0 \right)d\eta }\right)=\int\limits_{0}^{t}{{{\varphi }_{0}}\left( \eta  \right)\frac{\partial^2}{\partial x^2}{{G}_{\xi}}\left( t,x,\eta ,0 \right)d\eta };
\end{equation}
After certain evaluations (see Appendix A2) we have
$$\frac{\partial^2}{\partial x^2}{{G}_{\xi}}\left( t,x,\eta ,0 \right)=\sum_{n=-\infty}^{+\infty}sign(x+2na)\sum_{k=0}^{+\infty}\frac{(-1)^k\left|x+2na\right|^k(t-\eta)^{-\beta_1k-\beta-1}}{k!}E_{\alpha, -\beta_1k-\beta}^{-\gamma-\gamma_1k}\left[\delta(t-\eta)^{\alpha }\right].$$
If we substitute this into \eqref{eq2.24}, we see that \eqref{eq2.23} and \eqref{eq2.24} are equal. The term containing $\varphi_1(\eta)$ can be demonstrated analogously to the term containing $\varphi_0(\eta)$. 

Hence, we proceed to the term involving $\tau(\xi):$
$${}^{PC}D_{0t}^{\alpha ,\,\beta ,\,\gamma ,\,\delta }\left(\int\limits_{0}^{a}{\tau \left( \xi \right)\widetilde G\left( t,x,0,\xi \right)d\xi}\right)={}^{P}I_{0 t}^{\alpha, 1-\beta,-\gamma, \delta}\left(\int\limits_{0}^{a}{\tau \left( \xi \right)\widetilde G_t\left( t,x,0,\xi \right)d\xi}\right)=$$
$$=\int\limits_0^t(t-s)^{-\beta} E_{\alpha, 1-\beta}^{-\gamma}\left[\mathcal{\delta}(t-s)^\alpha\right]ds\int\limits_{0}^{a}{\tau \left( \xi \right)\widetilde G_s\left( s,x,0,\xi \right)d\xi}=$$
$$
   =\int\limits_0^a\tau \left( \xi \right)\,d\xi\int\limits_{0}^{t}{(t-s)^{-\beta} E_{\alpha, 1-\beta}^{-\gamma}\left[\mathcal{\delta}(t-s)^\alpha\right]\widetilde G_s\left( s,x,0,\xi \right)ds}. 
$$

After certain evaluations (see Appendix A3), we determine that 
$${}^{PC}D_{0t}^{\alpha ,\,\beta ,\,\gamma ,\,\delta }\left(\int\limits_{0}^{a}{\tau \left( \xi \right)\widetilde G\left( t,x,0,\xi \right)d\xi}\right)=\int\limits_{0}^{a}\tau(\xi)\times$$
\begin{equation}\label{eq2.25}
    \times\sum\limits_{n=-\infty }^{+\infty }{\sum\limits_{k=0}^{+\infty }{\frac{{{\left( -1 \right)}^{k}}\left[ {{\left| x-\xi+2na \right|}^{k}}-{{\left| x+\xi+2na \right|}^{k}} \right]}{2\Gamma(k+1)}}}t^{-\beta-\beta_1-\beta_{1}k}E_{\alpha, {1-\beta_1k-\beta_1-\beta}}^{-\gamma-\gamma _{1}-{{\gamma }_{1}}k}\left[\delta t^\alpha\right]d\xi.
\end{equation}

Now we compute the second-order partial derivative with respect to $x:$
\begin{equation}\label{eq2.26}
    \frac{\partial^2}{\partial x^2}\left(\int\limits_{0}^{a}{\tau \left( \xi \right)\widetilde G\left( t,x,0,\xi \right)d\xi}\right)=\int\limits_{0}^{a}{\tau \left( \xi \right)\frac{\partial^2}{\partial x^2}\widetilde G\left( t,x,0,\xi \right)d\xi};
\end{equation}
After certain evaluations (see Appendix A4), we will find that 
$$\frac{\partial^2}{\partial x^2}\widetilde G\left( t,x,0,\xi \right)=\sum\limits_{n=-\infty }^{+\infty }{\sum\limits_{k=0}^{+\infty }{\frac{{{\left( -1 \right)}^{k}}\left[ {{\left| x-\xi+2na \right|}^{k}}-{{\left| x+\xi+2na \right|}^{k}} \right]}{2\Gamma(k+1)}}}t^{-\beta-\beta_1-\beta_{1}k}E_{\alpha, {1-\beta_1k-\beta_1-\beta}}^{-\gamma-\gamma _{1}-{{\gamma }_{1}}k}\left[\delta t^\alpha\right].$$
Upon substituting the obtained result into \eqref{eq2.26} and comparing it with \eqref{eq2.25}, it becomes clear that both expressions are equal. We have proven that $\mathop{Lu_0=0}.$

To verify that $u_1(t,x)$ satisfies the equation \eqref{eq2.1}, we use the following relation:
$$ \mathop{Lu_1}=\mathop{L\int\limits_0^a\int\limits_0^t v(t,x,\eta,\xi)f(\eta,\xi) d\eta \, d\xi},$$
where \cite{Usmonov}
$$v\left( t,x;\eta ,\xi  \right)=\frac{{{\left( t-\eta  \right)}^{{{\beta }_{1}}-1}}}{2}{{E}_{12}}\left( \left. \begin{matrix}
   -{{\gamma }_{1}},1,{{\gamma }_{1}};\,\,\,\,\,\,\,\,\,\,\,\,\,\,\,\,\,\,\,\,\,\,\,\,\,\,\,  \\
   -{{\beta }_{1}},\alpha ,{{\beta }_{1}};-{{\gamma }_{1}},{{\gamma }_{1}};1,1;1,1;  \\
\end{matrix} \right|\begin{matrix}
   \left( -\left| x-\xi  \right| \right){{\left( t-\eta  \right)}^{-{{\beta }_{1}}}}  \\
   \delta {{\left( t-\eta  \right)}^{\alpha }}  \\
\end{matrix} \right).$$
In \cite{Usmonov}, the Cauchy problem was investigated for an equation similar to the equation \eqref{eq2.1}. It was proved there that the obtained solution satisfies the corresponding equation. Following the same approach as in \cite{Usmonov}, one can verify that $\mathop{Lu_1=f}.$

Next, we verify that the solution \eqref{eq2.4} satisfies the boundary conditions \eqref{eq2.2}. For this purpose, we let $x \to 0$ in \eqref{eq2.4}. From \eqref{eq2.5}, it follows that 
$$\lim _{x \to 0}{G}\left( t,x,\eta ,\xi \right)=0, \,\,\,\,\,\,\lim _{x \to 0}{G}_{\xi}\left( t,x,\eta ,a \right)=0.$$
According to \eqref{eq2.6}, we determine 
$$\lim _{x \to 0}\widetilde G\left( t,x,0,\xi \right)=0.$$
As a result, we get
$$\lim _{x \to 0}u\left( t,x \right)=\lim _{x \to 0}\int\limits_{0}^{t}{{{\varphi }_{0}}\left( \eta  \right){{G}_{\xi}}\left( t,x,\eta ,0 \right)d\eta }.$$
Now, let us evaluate this equality directly:
$$\lim _{x \to 0}\int\limits_{0}^{t}{{{\varphi }_{0}}\left( \eta  \right){{G}_{\xi}}\left( t,x,\eta ,0 \right)d\eta }=\lim _{x \to 0}\int\limits_{0}^{t}{{{\varphi }_{0}}\left( \eta  \right)\sum\limits_{n=-\infty }^{\infty }{sign\left( x+2na \right)\omega \left( t-\eta ,\left| x+2na \right| \right)\,}d\eta }=$$
$$=\lim _{x \to 0}\int\limits_{0}^{t}{{{\varphi }_{0}}\left( \eta  \right)\omega \left( t-\eta ,x \right)d\eta}-\lim _{x \to 0}\int\limits_{0}^{t}{{{\varphi }_{0}}\left( \eta  \right)\sum\limits_{n=1}^{\infty }{\left[ \omega \left( t-\eta ,2na-x \right)-\omega \left( t-\eta ,2na+x \right) \right]\,}d\eta }=$$
$$=\lim _{x \to 0}\int\limits_{0}^{t}{{{\varphi }_{0}}\left( \eta  \right)\omega \left( t-\eta ,x \right)d\eta}.$$

We replace the function $\omega$ with its value and apply the substitution 
$s=\frac{x}{(t-\eta)^{\beta_1}}$:
$$\lim _{x \to 0}u\left( t,x \right)=\lim _{x \to 0}\int\limits_{0}^{t}{{{\varphi }_{0}}\left( \eta  \right)\omega \left( t-\eta ,x \right)d\eta}=$$
$$=\lim _{x \to 0}\int\limits_{0}^{t}{{{\varphi }_{0}}\left( \eta  \right)\sum\limits_{n=0}^{+\infty }{\frac{{{\left( -1 \right)}^{n}}{{x}^{n}}}{n!}{{(t-\eta)}^{-{{\beta }_{1}}n-1}}E_{\alpha ,-{{\beta }_{1}}n}^{-{{\gamma }_{1}}n}\left[ \delta {{(t-\eta)}^{\alpha }} \right]}d\eta}=$$
$$=\lim _{x \to 0}\int\limits_{0}^{+\infty}{{{\varphi }_{0}}\left[ t-\left(\frac{x}{s}\right)^{\frac{1}{\beta_1}}  \right]\sum\limits_{n=0}^{+\infty }{\frac{{{\left( -1 \right)}^{n}}{{s}^{n-1}}}{\beta_1 n!}E_{\alpha ,-{{\beta }_{1}}n}^{-{{\gamma }_{1}}n}\left[ \delta \left(\frac{x}{s}\right)^{\frac{\alpha}{\beta_1}} \right]}ds}=$$
$$=\varphi_0(t)\int\limits_{0}^{+\infty}\sum\limits_{n=0}^{+\infty }{\frac{{{\left( -1 \right)}^{n}}{{s}^{n-1}}}{ n!\beta_1 \Gamma(-{{\beta }_{1}}n)}}ds=-\varphi_0(t)\left.\sum\limits_{n=0}^{+\infty }{\frac{{{\left( -1 \right)}^{n}}{{s}^{n}}}{ \Gamma(n+1)\Gamma(1-{{\beta }_{1}}n)}}\right|_0^{+\infty}=$$
$$=-\varphi_0(t)\lim_{s \to +\infty}e_{1,\beta_1}^{1,1}(-s)+\varphi_0(t)=\varphi_0(t),\,\,\,\,0\le t\le T.$$

Similarly, we have
$$\lim _{x \to a}u\left( t,x \right)=\varphi_1(t), \,\,\,\, 0 \le t\le T.$$

We now check that our solution \eqref{eq2.4} fulfills the initial condition \eqref{eq2.3}:
$$\lim_{t \to 0} u\left( t,x \right)=\lim_{t \to 0}\int\limits_{0}^{a}{\tau \left( \xi \right)\widetilde G\left( t,x,0,\xi \right)d\xi}=$$
$$=\lim_{t \to 0}\frac{t^{-\beta_1}}{2}\sum\limits_{n=-\infty }^{+\infty }{\sum\limits_{k=0}^{+\infty }{\frac{{{\left( -1 \right)}^{k}}}{k!}}}\sum_{i=0}^{+\infty}\frac{{(-{\gamma }_{1}}-{{\gamma }_{1}}k)_{i}\delta^it^{\alpha i}}{i!\,\Gamma(\alpha i+1-\beta_1-\beta_1k)}\times$$
$$\times\,\int\limits_{0}^{a}\tau(\xi)\frac{\left[ {{\left| x-\xi+2na \right|}^{k}}-{{\left| x+\xi+2na \right|}^{k}} \right]}{t^{\beta_1k}}d\xi=$$
$$=\lim_{t \to 0}\frac{t^{-\beta_1}}{2}{\sum\limits_{k=0}^{+\infty }}\sum_{i=0}^{+\infty}\frac{{\left( -1 \right)}^{k}{(-{\gamma }_{1}}-{{\gamma }_{1}}k)_{i}\delta^it^{\alpha i}}{k!i!\,\Gamma(\alpha i+1-\beta_1-\beta_1k)}\int\limits_{0}^{a}\tau(\xi)\frac{{\left| x-\xi \right|}^{k}}{t^{\beta_1k}}d\xi-$$
$$-\lim_{t \to 0}\frac{t^{-\beta_1}}{2}{\sum\limits_{k=0}^{+\infty }}\sum_{i=0}^{+\infty}\frac{{\left( -1 \right)}^{k}{(-{\gamma }_{1}}-{{\gamma }_{1}}k)_{i}\delta^it^{\alpha i}}{k!i!\,\Gamma(\alpha i+1-\beta_1-\beta_1k)}\int\limits_{0}^{a}\tau(\xi)\frac{{\left| x+\xi \right|}^{k}}{t^{\beta_1k}}d\xi+$$
$$+\lim_{t \to 0}\frac{t^{-\beta_1}}{2}\left(\sum\limits_{n=-\infty }^{-1 }+\sum\limits_{n=1 }^{+\infty }\right){\sum\limits_{k=0}^{+\infty }{\frac{{{\left( -1 \right)}^{k}}}{k!}}}\sum_{i=0}^{+\infty}\frac{{(-{\gamma }_{1}}-{{\gamma }_{1}}k)_{i}\delta^it^{\alpha i}}{i!\,\Gamma(\alpha i+1-\beta_1-\beta_1k)}\times$$
\begin{equation}\label{eq2.27}
    \times\,\int\limits_{0}^{a}\tau(\xi)\frac{\left[ {{\left| x-\xi+2na \right|}^{k}}-{{\left| x+\xi+2na \right|}^{k}} \right]}{t^{\beta_1k}}d\xi.
\end{equation}
The first term in \eqref{eq2.27} can be written as follows:
$$\lim_{t \to 0}\frac{t^{-\beta_1}}{2}{\sum\limits_{k=0}^{+\infty }}\sum_{i=0}^{+\infty}\frac{{\left( -1 \right)}^{k}{(-{\gamma }_{1}}-{{\gamma }_{1}}k)_{i}\delta^it^{\alpha i}}{k!i!\,\Gamma(\alpha i+1-\beta_1-\beta_1k)}\left[\int\limits_{0}^{x}\tau(\xi)\frac{{\left( x-\xi \right)}^{k}}{t^{\beta_1k}}d\xi+\int\limits_{x}^{a}\tau(\xi)\frac{{\left (\xi-x \right)}^{k}}{t^{\beta_1k}}d\xi\right].$$
To simplify the analysis, we evaluate only the term containing the first integral and perform the substitution $z=\frac{x-\xi}{t^{\beta_1}}$, then use $$\lim_{|z| \to \infty} ze_{\alpha,\beta}^{\mu,\delta}(z)=-\frac{1}{\Gamma(\mu-\alpha)\Gamma(\delta+\beta)}$$
to get
$$\lim_{t \to 0}\frac{t^{-\beta_1}}{2}{\sum\limits_{k=0}^{+\infty }}\sum_{i=0}^{+\infty}\frac{{\left( -1 \right)}^{k}{(-{\gamma }_{1}}-{{\gamma }_{1}}k)_{i}\delta^it^{\alpha i}}{k!i!\,\Gamma(\alpha i+1-\beta_1-\beta_1k)}\int\limits_{0}^{x}\tau(\xi)\frac{{\left( x-\xi \right)}^{k}}{t^{\beta_1k}}d\xi=$$
$$=\lim_{t \to 0}\frac{1}{2}{\sum\limits_{k=0}^{+\infty }}\sum_{i=0}^{+\infty}\frac{{\left( -1 \right)}^{k}{(-{\gamma }_{1}}-{{\gamma }_{1}}k)_{i}\delta^it^{\alpha i}}{k!i!\,\Gamma(\alpha i+1-\beta_1-\beta_1k)}\int\limits_{0}^{+\infty}\tau(x-t^{\beta_1}z)z^k\,dz=$$
$$=\left.\frac{\tau(x)}{2}{\sum\limits_{k=0}^{+\infty }}\frac{{z\left( -z \right)}^{k}}{\Gamma(k+2)\,\Gamma(1-\beta_1-\beta_1k)}\right|_0^{+\infty}=-\frac{\tau(x)}{2}\lim_{z \to +\infty}(-z)e_{1,\beta_1}^{2,1-\beta_1}(-z)=\frac{\tau(x)}{2}.$$
In the same way, we determine that 
$$\lim_{t \to 0}\frac{t^{-\beta_1}}{2}{\sum\limits_{k=0}^{+\infty }}\sum_{i=0}^{+\infty}\frac{{\left( -1 \right)}^{k}{(-{\gamma }_{1}}-{{\gamma }_{1}}k)_{i}\delta^it^{\alpha i}}{k!i!\,\Gamma(\alpha i+1-\beta_1-\beta_1k)}\int\limits_{x}^{a}\tau(\xi)\frac{{\left( \xi-x \right)}^{k}}{t^{\beta_1k}}d\xi=\frac{\tau(x)}{2}.$$
Applying the same procedure to the other terms in equality \eqref{eq2.27}, we obtain that their values are equal to zero. Consequently, it follows that 
$$\lim_{t \to 0} u\left( t,x \right)=\tau(x),\,\,\,\,0\le x \le a.$$

Theorem 2.1 is proved.

\section{Examples and numerical illustrations}
In this section, we present 2 examples considering the influence of initial data and external source separately. 

The following example shows the influence of initial data on the process. In this case, we assume that no external force is involved in the process. 

\textbf{Example 1.} Let $a=\pi,$ $T=2,$ $\varphi_0(t)=\varphi_1(t)=0,$ $f(t,x)=0,$ $\tau(x)=sin(x).$ In this case, from \eqref{eq2.4} one can easily get 
$$u(t,x)=\int\limits_{0}^{\pi}{sin \left( \xi \right)\widetilde G(t,x,0,\xi)d\xi},$$
where
$$\widetilde G(t,x,0,\xi)=\frac{1}{2}\sum\limits_{n=-\infty }^{+\infty }{\sum\limits_{k=0}^{+\infty }{\frac{{{\left( -1 \right)}^{k}}\left[ {{\left| x-\xi+2n\pi \right|}^{k}}-{{\left| x+\xi+2n\pi \right|}^{k}} \right]}{k!}}}\times$$
$$\times\sum_{i=0}^{+\infty}\frac{\delta^it^{\alpha i-{{\beta }_{1}}-\beta_1k}\Gamma(i{-{\gamma }_{1}}-{{\gamma }_{1}}k)}{\Gamma(i+1)\,\Gamma({-{\gamma }_{1}}-{{\gamma }_{1}}k)\Gamma(\alpha i+1-\beta_1-\beta_1k)},$$
$x\in[0,\pi],$ $t\in[0,2]$, $\beta_1=\frac{\beta}{2}$, $\gamma_1=\frac{\gamma}{2}$.

We illustrate the influence of fractional order $\beta$ on the solution behavior. We illustrate in graphs the fixation of the parameters as follows: 

{\bf (1)} $\alpha=0,8$, $\beta=0,1$, $\gamma=0,3$, $\delta=0,5$; 

{\bf (2)} $\alpha=0,8$, $\beta=0,5$, $\gamma=0,3$, $\delta=0,5$; 

{\bf (3)} $\alpha=0,8$, $\beta=0,9$, $\gamma=0,3$, $\delta=0,5$.

\begin{figure}[h]
    \centering
    \includegraphics[width=0.99\linewidth]{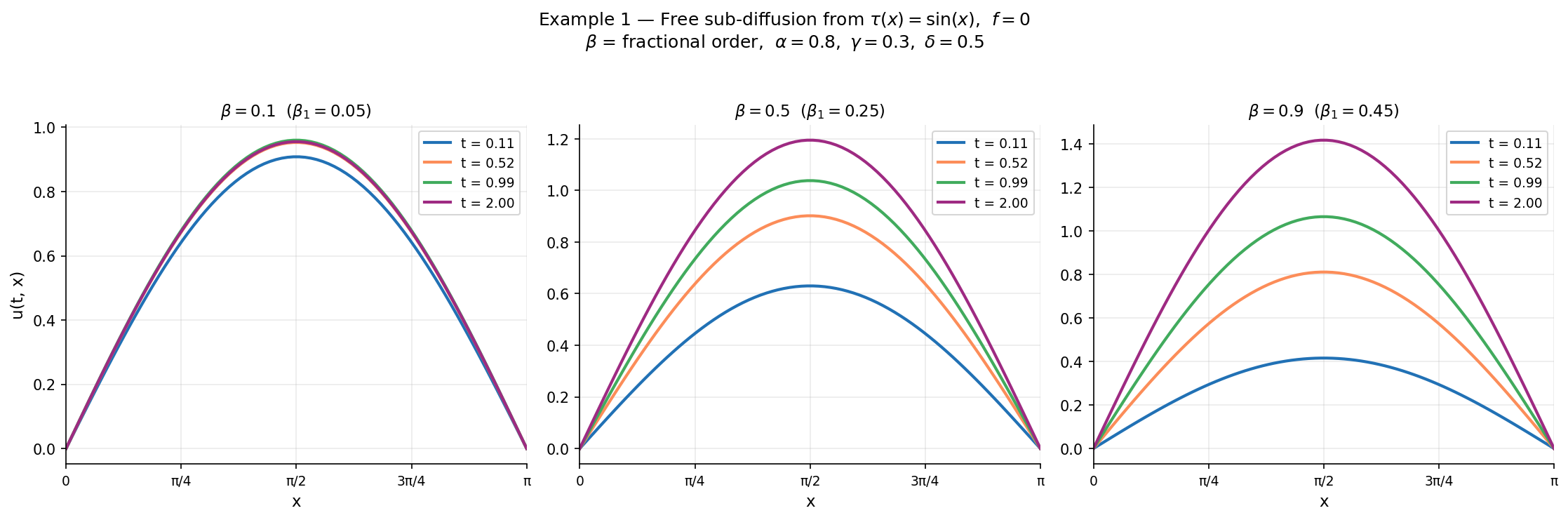}
    \caption{Influence of the initial data.}
    \label{fig:1}
\end{figure}

{\it Explanation:} The initial sin(x) profile relaxes toward zero with no external forcing. The crucial corrected behaviour across the three panels is:
\begin{itemize}
\item $\beta = 0.1$\, ($\beta_1 = 0.05$, very small fractional order): extremely slow decay. The small fractional order means almost no memory of the driving force dissipates over time. The profile barely decreases even at $t = 2$; 

\item $\beta = 0.5$\, ($\beta_1 = 0.25$): intermediate decay rate;  

\item $\beta= 0.9$\, ($\beta_1= 0.45$): fastest decay, approaching half-order classical behaviour. The larger fractional order brings the process closer to normal diffusion, so the profile relaxes much more quickly.
\end{itemize}
The next example deals with the influence of the external force. That is, we assume that there is no influence of the initial data ($u(0,x)\equiv 0$). 

\textbf{Example 2.} Let $a=\pi$, $T=2$, $\varphi_0(t)=\varphi_1(t)=0,$ $\tau(x)=0,$ $f(t,x)=t\sin x$. Then according to \eqref{eq2.4} we get
$$u(t,x)=\int\limits_{0}^{t}{\int\limits_{0}^{\pi}{\eta sin(\xi) G\left( t,x,\eta ,\xi \right)d\xi d\eta }},$$
where
$$G(t,x,\eta,\xi)=\frac{1}{2}{{(t-\eta)}^{{{\beta }_{1}}-1}}\times$$
$$\times\sum\limits_{n=-\infty }^{+\infty }{\sum\limits_{k=0}^{+\infty }{\sum\limits_{m=0}^{+\infty }}}\left[{{{\frac{\Gamma \left( {{\gamma }_{1}}-{{\gamma }_{1}}k+m \right){{\left[ -\left| x-\xi+2n\pi \right|{{(t-\eta)}^{-{{\beta }_{1}}}} \right]}^{k}}{{\left[ \delta {{(t-\eta)}^{\alpha }} \right]}^{m}}}{\Gamma \left( k+1 \right)\Gamma \left( {{\gamma }_{1}}-{{\gamma }_{1}}k \right)\Gamma \left( \alpha m+{{\beta }_{1}}-{{\beta }_{1}}k \right)\Gamma \left( m+1 \right)}}-}}\right.$$
$$\left.-{{{\frac{\Gamma \left( {{\gamma }_{1}}-{{\gamma }_{1}}k+m \right){{\left[ -\left| x+\xi+2n\pi \right|{{(t-\eta)}^{-{{\beta }_{1}}}} \right]}^{k}}{{\left[ \delta {{(t-\eta)}^{\alpha }} \right]}^{m}}}{\Gamma \left( k+1 \right)\Gamma \left( {{\gamma }_{1}}-{{\gamma }_{1}}k \right)\Gamma \left( \alpha m+{{\beta }_{1}}-{{\beta }_{1}}k \right)\Gamma \left( m+1 \right)}}}}\right],$$
$x\in[0,\pi],$ $t\in[0,2]$,  $\beta_1=\frac{\beta}{2}$, $\gamma_1=\frac{\gamma}{2}$.

We illustrate in graphs the fixation of the parameters as follows:

{\bf (1)} $\alpha=0,8$, $\beta=0,1$, $\gamma=0,3$, $\delta=0,5$; 

{\bf (2)} $\alpha=0,8$, $\beta=0,5$, $\gamma=0,3$, $\delta=0,5$; 

{\bf (3)} $\alpha=0,8$, $\beta=0,9$, $\gamma=0,3$, $\delta=0,5$.

\begin{figure}[h]
    \centering
    \includegraphics[width=0.99\linewidth]{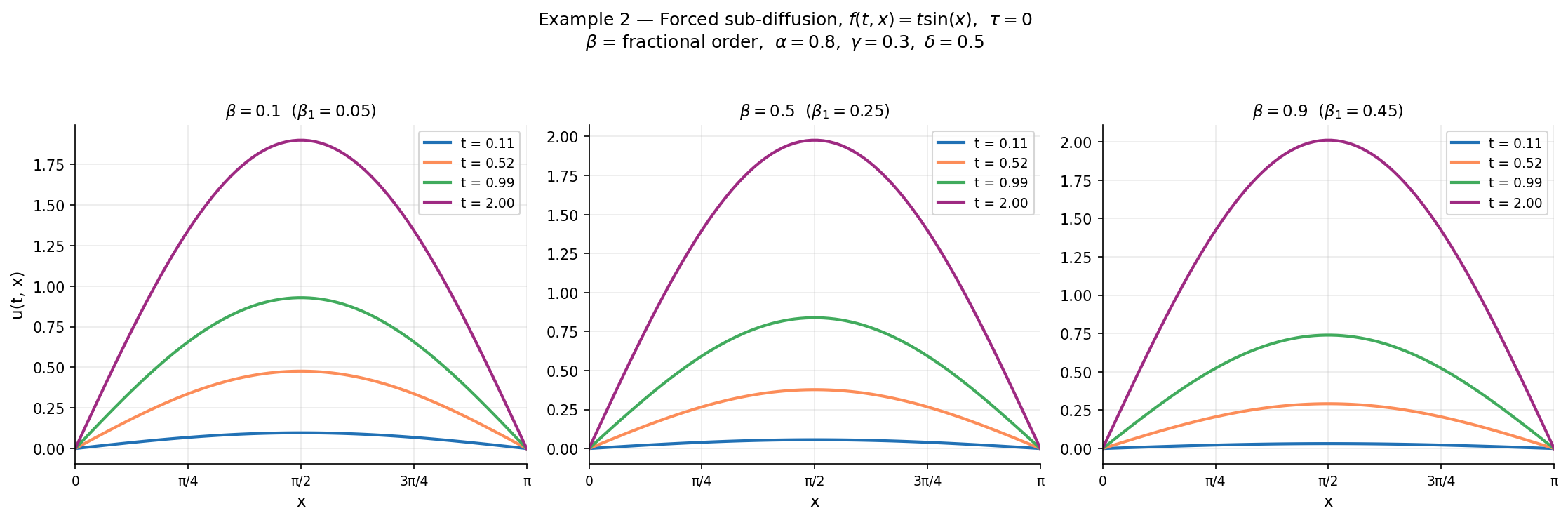}
    \caption{Influence of the external force.}
    \label{fig:2}
\end{figure}

{\it Explanation:} With $f(t,x) = t\sin x$ pumping energy in and zero initial conditions, the field accumulates from zero. The corrected trend:
\begin{itemize}
\item $\beta = 0.1$: the response to forcing is very slow and the amplitude remains small -- deeply sub-diffusive systems respond sluggishly to sources.

\item $\beta = 0.9$: the amplitude grows much faster and reaches the highest values by $t = 2$, because the larger fractional order allows the system to respond more efficiently to the accumulating source.
\end{itemize}
The parameters $\alpha$ and $\delta$ (fixed here at $0.8$ and $0.5$) control the shape and timescale of the Prabhakar kernel, modulating the overall magnitude and curvature of the time evolution without changing this fundamental monotone trend in $\beta$.

\section{Conclusion}
In this work, we have successfully established a Green’s function framework for solving the first initial-boundary value problem for a sub-diffusion equation involving the regularized Prabhakar fractional derivative. By utilizing the superposition method, we reduced the complex non-homogeneous problem into manageable components, allowing for the derivation of an explicit solution representation. A significant result of this study is the explicit construction of Green's function, which we have expressed in terms of a bivariate Mittag-Leffler-type function $E_{12}$. This formulation aligns with the broader development of bivariate fractional calculus and operational methods, such as those explored in the works of Simon Isah regarding general analytic kernels and Noosheza Rani concerning the algebraic structures of Prabhakar-type operators. Our rigorous proof confirms that the obtained representation indeed constitutes a regular solution, satisfying all imposed boundary and initial conditions. 

The methodology presented here further generalizes the structural properties of generalized fractional derivatives, a field significantly advanced by Tomovski through the study of Hilfer-Prabhakar derivatives and their related inequalities. These results provide a robust basis for further research into anomalous diffusion models and hereditary systems where memory effects require the sophisticated parameterization offered by the Prabhakar kernel. Future developments may include extending this framework to multi-dimensional domains or investigating the influence of varying fractional orders on the stability of such systems.

\section{Statements and Declarations}
\subsection{Data Availability Statement}
No new data were created or analyzed in this study. Therefore, data sharing is not applicable.

\subsection{Funding Statement}
The authors are partially supported by Methusalem programme of the Ghent University Special Research Fund(BOF), grant/award number: 01M01021 and the FWO Odysseus 1, grant/award number: G.0H94.18N.
\subsection{Conflict of Interest Disclosure}
The authors declare that there are no conflicts of interest regarding the publication of this paper.

\section{Appendices}
\subsection{Appendix A1.}
Let us compute the integral with respect to $s$ separately:
$$\int\limits_{\eta}^{t}(t-s)^{-\beta} E_{\alpha, 1-\beta}^{-\gamma}\left[\mathcal{\delta}(t-s)^\alpha\right]{{G}_{\xi}}\left( s,x,\eta ,0 \right)ds=\sum_{n=-\infty}^{+\infty}sign(x+2na)\sum_{k=0}^{+\infty}\frac{(-1)^k\left|x+2na\right|^k}{k!}\times$$
$$\times\int\limits_{\eta}^{t}(t-s)^{-\beta} (s-\eta)^{-\beta_1k-1}E_{\alpha, 1-\beta}^{-\gamma}\left[\mathcal{\delta}(t-s)^\alpha\right]E_{\alpha, -\beta_1k}^{-\gamma_1k}\left[\mathcal{\delta}(s-\eta)^\alpha\right]ds.$$
Next, we use the formula  \cite{Knopp}
\begin{equation}\label{eq2.21}
	\sum\limits_{k=0}^{+\infty }{{{a}_{k}}}\sum\limits_{m=0}^{+\infty }{{{b}_{m}}=\sum\limits_{k=0}^{+\infty }{\sum\limits_{m=0}^{k}{{{a}_{m}}}{{b}_{k-m}}}}
\end{equation}
for the generalized Mittag-Leffler functions and perform the substitution $s=(t-\eta)z+\eta$:
$$\sum_{n=-\infty}^{+\infty}sign(x+2na)\sum_{k=0}^{+\infty}\frac{(-1)^k\left|x+2na\right|^k}{k!}\sum_{i=0}^{+\infty}\sum_{j=0}^{i}\frac{\delta^i(-\gamma)_j(-\gamma_1k)_{i-j}}{j!(i-j)!\Gamma(\alpha j+1-\beta)\Gamma(\alpha i-\alpha j-\beta_1k)}\times$$
$$\times\int\limits_{\eta}^{t}(t-s)^{\alpha j-\beta} (s-\eta)^{\alpha i-\alpha j-\beta_1k-1}ds=$$
$$=\sum_{n=-\infty}^{+\infty}sign(x+2na)\sum_{k=0}^{+\infty}\frac{(-1)^k\left|x+2na\right|^k}{k!}\sum_{i=0}^{+\infty}\frac{\delta^i (t-\eta)^{\alpha i-\beta_1k-\beta}}{\Gamma(\alpha i-\beta_1k-\beta+1)}\sum_{j=0}^{i}\frac{(-\gamma)_j(-\gamma_1k)_{i-j}}{j!(i-j)!}.$$
In accordance with the formula \cite{Prudnikov}
\begin{equation}\label{eq2.22}
	\sum\limits_{m=0}^{k}{\frac{{{\left( \delta  \right)}_{m}}{{\left( \gamma  \right)}_{k-m}}}{m!\left( k-m \right)!}=\frac{{{\left( \delta +\gamma  \right)}_{k}}}{k!}},
\end{equation}
we obtain
$$\sum_{n=-\infty}^{+\infty}sign(x+2na)\sum_{k=0}^{+\infty}\frac{(-1)^k\left|x+2na\right|^k(t-\eta)^{-\beta_1k-\beta}}{k!}\sum_{i=0}^{+\infty}\frac{ (-\gamma-\gamma_1k)_i\delta^i(t-\eta)^{\alpha i}}{i!\Gamma(\alpha i-\beta_1k-\beta+1)}=$$
$$=\sum_{n=-\infty}^{+\infty}sign(x+2na)\sum_{k=0}^{+\infty}\frac{(-1)^k\left|x+2na\right|^k(t-\eta)^{-\beta_1k-\beta}}{k!}E_{\alpha, 1-\beta_1k-\beta}^{-\gamma-\gamma_1k}\left[\delta(t-\eta)^{\alpha }\right].$$

Next, we insert the derived result into the preceding equality and proceed with the computation:
$$\frac{\partial}{\partial t}\int\limits_0^t{{\varphi }_{0}}\left( \eta  \right)\sum_{n=-\infty}^{+\infty}sign(x+2na)\sum_{k=0}^{+\infty}\frac{(-1)^k\left|x+2na\right|^k(t-\eta)^{-\beta_1k-\beta}}{k!}E_{\alpha, 1-\beta_1k-\beta}^{-\gamma-\gamma_1k}\left[\delta(t-\eta)^{\alpha }\right]d\eta=$$
$$=\lim_{\eta \to t}\left[{{\varphi }_{0}}\left( \eta  \right)\sum_{n=-\infty}^{+\infty}sign(x+2na)\sum_{k=0}^{+\infty}\frac{(-1)^k\left|x+2na\right|^k(t-\eta)^{-\beta_1k-\beta}}{k!}E_{\alpha, 1-\beta_1k-\beta}^{-\gamma-\gamma_1k}\left[\delta(t-\eta)^{\alpha }\right]\right]+$$
$$+\int\limits_0^t{{\varphi }_{0}}\left( \eta  \right)\sum_{n=-\infty}^{+\infty}sign(x+2na)\sum_{k=0}^{+\infty}\frac{(-1)^k\left|x+2na\right|^k(t-\eta)^{-\beta_1k-\beta-1}}{k!}E_{\alpha, -\beta_1k-\beta}^{-\gamma-\gamma_1k}\left[\delta(t-\eta)^{\alpha }\right]d\eta.$$
Now let us carefully evaluate the limit that arises from taking the derivative:
$$\lim_{\eta \to t}\left[{{\varphi }_{0}}\left( \eta  \right)\sum_{n=-\infty}^{+\infty}sign(x+2na)\sum_{k=0}^{+\infty}\frac{(-1)^k\left|x+2na\right|^k(t-\eta)^{-\beta_1k-\beta}}{k!}\right]\times$$
$$\times\lim_{\eta \to t}E_{\alpha, 1-\beta_1k-\beta}^{-\gamma-\gamma_1k}\left[\delta(t-\eta)^{\alpha }\right]=$$
$$=\lim_{\eta \to t}\left[{{\varphi }_{0}}\left( \eta  \right)\sum_{n=-\infty}^{+\infty}sign(x+2na)(t-\eta)^{-\beta}\sum_{k=0}^{+\infty}\frac{(-1)^k\left|x+2na\right|^k(t-\eta)^{-\beta_1k}}{k!\Gamma(1-\beta_1k-\beta)}\right].$$
We rewrite the final expression using the Wright-type function \cite{Pskhu 2}  
$$e_{\alpha,{\beta}}^{\mu,{\delta }}\left( z \right)=\sum_{n=0}^{+\infty}\frac{z^n}{\Gamma(\alpha n+\mu)\Gamma(\delta-\beta n)}$$  
and make use of its  property \cite{Pskhu 2}
$$\lim_{\left|z\right| \to \infty}z^2 e_{\alpha,{\beta}}^{\alpha-k,{\delta }}\left( z \right)=\frac{1}{\Gamma(-k-\alpha)\Gamma(\delta+2\beta)}:$$

$$\lim_{\eta \to t}\left[{{\varphi }_{0}}\left( \eta  \right)\sum_{n=-\infty}^{+\infty}sign(x+2na)(t-\eta)^{-\beta}e_{1,\beta_1}^{1,1-\beta}\left(-\frac{\left|x+2na\right|}{(t-\eta)^{\beta_1}}\right)\right]=$$
$$=\lim_{\eta \to t}\left[{{\varphi }_{0}}\left( \eta  \right)\sum_{n=-\infty}^{+\infty}\frac{sign(x+2na)}{(x+2na)^2}\left(-\frac{\left|x+2na\right|}{(t-\eta)^{\beta_1}}\right)^2e_{1,\beta_1}^{1,1-\beta}\left(-\frac{\left|x+2na\right|}{(t-\eta)^{\beta_1}}\right)\right]=$$
$$=\frac{{{\varphi }_{0}}\left( t  \right)}{\Gamma(-1)\Gamma(1)}\sum_{n=-\infty}^{+\infty}\frac{sign(x+2na)}{(x+2na)^2}=0.$$

\subsection{Appendix A2.}
Let us make the following simplifications:
$$\frac{\partial^2}{\partial x^2}{{G}_{\xi}}\left( t,x,\eta ,0 \right)=\frac{\partial^2}{\partial x^2}\sum_{n=-\infty}^{+\infty}sign(x+2na)\omega(t-\eta,\left|x+2na\right|)=$$
$$=\frac{\partial^2}{\partial x^2}\left[-\sum_{n=-\infty}^{-1}\omega(t-\eta,-x-2na)+\sum_{n=0}^{+\infty}\omega(t-\eta,x+2na)\right]=$$
$$=\frac{\partial^2}{\partial x^2}\sum_{n=-\infty}^{-1}\left(-(t-\eta)^{-1}E_{\alpha, 0}^{0}\left[ \delta {{(t-\eta)}^{\alpha }} \right]-\sum\limits_{k=1}^{+\infty }{\frac{{}{{(x+2na)}^{k}}}{k!}{{(t-\eta)}^{-{{\beta }_{1}}k-1}}E_{\alpha ,-{{\beta }_{1}}k}^{-{{\gamma }_{1}}k}\left[ \delta {{(t-\eta)}^{\alpha }} \right]}\right)+$$
$$+\frac{\partial^2}{\partial x^2}\sum_{n=0}^{+\infty}\left((t-\eta)^{-1}E_{\alpha, 0}^{0}\left[ \delta {{(t-\eta)}^{\alpha }} \right]+\sum\limits_{k=1}^{+\infty }{\frac{{{\left( -1 \right)}^{k}}{{(x+2na)}^{k}}}{k!}{{(t-\eta)}^{-{{\beta }_{1}}k-1}}E_{\alpha ,-{{\beta }_{1}}k}^{-{{\gamma }_{1}}k}\left[ \delta {{(t-\eta)}^{\alpha }} \right]}\right)=$$
$$=\frac{\partial}{\partial x}\sum_{n=-\infty}^{-1}\left(-\sum\limits_{k=0}^{+\infty }{\frac{{}{{(x+2na)}^{k}}}{k!}{{(t-\eta)}^{-{{\beta }_{1}}k-\beta_1-1}}E_{\alpha ,-{{\beta }_{1}}k-\beta_1}^{-{{\gamma }_{1}}k-\gamma_1}\left[ \delta {{(t-\eta)}^{\alpha }} \right]}\right)+$$
$$+\frac{\partial}{\partial x}\sum_{n=0}^{+\infty}\left(-\sum\limits_{k=0}^{+\infty }{\frac{{{\left( -1 \right)}^{k}}{{(x+2na)}^{k}}}{k!}{{(t-\eta)}^{-{{\beta }_{1}}k-\beta_1-1}}E_{\alpha ,-{{\beta }_{1}}k-\beta_1}^{-{{\gamma }_{1}}k-\gamma_1}\left[ \delta {{(t-\eta)}^{\alpha }} \right]}\right)=$$
$$=\frac{\partial}{\partial x}\sum_{n=-\infty}^{-1}\left(-{{(t-\eta)}^{-\beta_1-1}}E_{\alpha ,-\beta_1}^{-\gamma_1}\left[ \delta {{(t-\eta)}^{\alpha }} \right]-\right.$$
$$\left.-\sum\limits_{k=1}^{+\infty }{\frac{{}{{(x+2na)}^{k}}}{k!}{{(t-\eta)}^{-{{\beta }_{1}}k-\beta_1-1}}E_{\alpha ,-{{\beta }_{1}}k-\beta_1}^{-{{\gamma }_{1}}k-\gamma_1}\left[ \delta {{(t-\eta)}^{\alpha }} \right]}\right)+$$
$$+\frac{\partial}{\partial x}\sum_{n=0}^{+\infty}\left(-{{(t-\eta)}^{-\beta_1-1}}E_{\alpha ,-\beta_1}^{-\gamma_1}\left[ \delta {{(t-\eta)}^{\alpha }} \right]-\right.$$
$$\left.-\sum\limits_{k=1}^{+\infty }{\frac{{{\left( -1 \right)}^{k}}{{(x+2na)}^{k}}}{k!}{{(t-\eta)}^{-{{\beta }_{1}}k-\beta_1-1}}E_{\alpha ,-{{\beta }_{1}}k-\beta_1}^{-{{\gamma }_{1}}k-\gamma_1}\left[ \delta {{(t-\eta)}^{\alpha }} \right]}\right)=$$
$$=-\sum_{n=-\infty}^{-1}\sum\limits_{k=0}^{+\infty }{\frac{{}{{(x+2na)}^{k}}}{k!}{{(t-\eta)}^{-{{\beta }_{1}}k-2\beta_1-1}}E_{\alpha ,-{{\beta }_{1}}k-2\beta_1}^{-{{\gamma }_{1}}k-2\gamma_1}\left[ \delta {{(t-\eta)}^{\alpha }} \right]}+$$
$$+\sum_{n=0}^{+\infty}\sum\limits_{k=0}^{+\infty }{\frac{{{\left( -1 \right)}^{k}}{{(x+2na)}^{k}}}{k!}{{(t-\eta)}^{-{{\beta }_{1}}k-2\beta_1-1}}E_{\alpha ,-{{\beta }_{1}}k-2\beta_1}^{-{{\gamma }_{1}}k-2\gamma_1}\left[ \delta {{(t-\eta)}^{\alpha }} \right]}=$$
$$=\sum_{n=-\infty}^{+\infty}sign(x+2na)\sum_{k=0}^{+\infty}\frac{(-1)^k\left|x+2na\right|^k(t-\eta)^{-\beta_1k-\beta-1}}{k!}E_{\alpha, -\beta_1k-\beta}^{-\gamma-\gamma_1k}\left[\delta(t-\eta)^{\alpha }\right].$$

\subsection{Appendix A3.}

First, let us compute the value of $\widetilde G\left( t,x,0,\xi \right).$ For this purpose, we substitute the explicit form of the function 
$G\left( t,x,\eta ,\xi \right)$ into \eqref{eq2.6}:
$$\widetilde G\left( t,x,0,\xi \right)=\int\limits_{0}^{t}\eta^{-\beta}E_{\alpha, 1-\beta}^{-\gamma}[\delta \eta^\alpha]{\frac{{{(t-\eta)}^{{{\beta }_{1}}-1}}}{2}\times}$$
$${{\times\sum\limits_{n=-\infty }^{+\infty }{\sum\limits_{k=0}^{+\infty }{\frac{{{\left( -1 \right)}^{k}}\left[ {{\left| x-\xi+2na \right|}^{k}}-{{\left| x+\xi+2na \right|}^{k}} \right]}{k!}{{(t-\eta)}^{-{{\beta }_{1}}k}}E_{\alpha ,{{\beta }_{1}}-{{\beta }_{1}}k}^{{{\gamma }_{1}}-{{\gamma }_{1}}k}\left[ \delta {{(t-\eta)}^{\alpha }} \right]}} \,d\eta }}=$$
$$=\frac{1}{2}\sum\limits_{n=-\infty }^{+\infty }{\sum\limits_{k=0}^{+\infty }{\frac{{{\left( -1 \right)}^{k}}\left[ {{\left| x-\xi+2na \right|}^{k}}-{{\left| x+\xi+2na \right|}^{k}} \right]}{k!}}}\times$$
$$\times\int\limits_{0}^{t}\eta^{-\beta}E_{\alpha, 1-\beta}^{-\gamma}[\delta \eta^\alpha]{{(t-\eta)}^{{{\beta }_{1}}-\beta_1k-1}}E_{\alpha ,{{\beta }_{1}}-{{\beta }_{1}}k}^{{{\gamma }_{1}}-{{\gamma }_{1}}k}\left[ \delta {{(t-\eta)}^{\alpha }} \right]\,d\eta.$$
We utilize \eqref{eq2.21} for $E_{\alpha, 1-\beta}^{-\gamma}[\delta \eta^\alpha]$ and $E_{\alpha ,{{\beta }_{1}}-{{\beta }_{1}}k}^{{{\gamma }_{1}}-{{\gamma }_{1}}k}\left[ \delta {{(t-\eta)}^{\alpha }} \right]$:
$$\widetilde G\left( t,x,0,\xi \right)=\frac{1}{2}\sum\limits_{n=-\infty }^{+\infty }{\sum\limits_{k=0}^{+\infty }{\frac{{{\left( -1 \right)}^{k}}\left[ {{\left| x-\xi+2na \right|}^{k}}-{{\left| x+\xi+2na \right|}^{k}} \right]}{k!}}}\times$$
$$\times\int\limits_{0}^{t}\sum_{i=0}^{+\infty}\sum_{j=0}^{i}\frac{\delta^i(-\gamma)_j({{\gamma }_{1}}-{{\gamma }_{1}}k)_{i-j}}{j!(i-j)!\,\Gamma(\alpha j+1-\beta)\Gamma(\alpha i-\alpha j+\beta_1-\beta_1k)}\eta^{\alpha j-\beta}{{(t-\eta)}^{\alpha i-\alpha j+{{\beta }_{1}}-\beta_1k-1}}\,d\eta.$$
Then we make the substitution $\eta=ts$ and apply the formula \eqref{eq2.22}:
$$\widetilde G\left( t,x,0,\xi \right)=\frac{1}{2}\sum\limits_{n=-\infty }^{+\infty }{\sum\limits_{k=0}^{+\infty }{\frac{{{\left( -1 \right)}^{k}}\left[ {{\left| x-\xi+2na \right|}^{k}}-{{\left| x+\xi+2na \right|}^{k}} \right]}{k!}}}\times$$
$$\times\sum_{i=0}^{+\infty}\sum_{j=0}^{i}\frac{\delta^it^{\alpha i-{{\beta }_{1}}-\beta_1k}(-\gamma)_j({{\gamma }_{1}}-{{\gamma }_{1}}k)_{i-j}}{j!(i-j)!\,\Gamma(\alpha i+1-\beta_1-\beta_1k)}=$$
$$=\frac{1}{2}\sum\limits_{n=-\infty }^{+\infty }{\sum\limits_{k=0}^{+\infty }{\frac{{{\left( -1 \right)}^{k}}\left[ {{\left| x-\xi+2na \right|}^{k}}-{{\left| x+\xi+2na \right|}^{k}} \right]}{k!}}}\sum_{i=0}^{+\infty}\frac{\delta^it^{\alpha i-{{\beta }_{1}}-\beta_1k}({-{\gamma }_{1}}-{{\gamma }_{1}}k)_{i}}{i!\,\Gamma(\alpha i+1-\beta_1-\beta_1k)}=$$
$$=\frac{1}{2}\sum\limits_{n=-\infty }^{+\infty }{\sum\limits_{k=0}^{+\infty }{\frac{{{\left( -1 \right)}^{k}}\left[ {{\left| x-\xi+2na \right|}^{k}}-{{\left| x+\xi+2na \right|}^{k}} \right]}{k!}}}\times$$
$$\times\sum_{i=0}^{+\infty}\frac{\delta^it^{\alpha i-{{\beta }_{1}}-\beta_1k}\Gamma(i{-{\gamma }_{1}}-{{\gamma }_{1}}k)}{\Gamma(i+1)\,\Gamma({-{\gamma }_{1}}-{{\gamma }_{1}}k)\Gamma(\alpha i+1-\beta_1-\beta_1k)}.$$
Using the obtained expression, we represent the function 
$\widetilde G\left( t,x,0,\xi \right)$ in terms of $E_{12}:$
$$\widetilde G\left( t,x,0,\xi \right)=\frac{{{ t }^{{{-\beta }_{1}}}}}{2}\sum\limits_{n=-\infty }^{\infty }{\left[ {{E}_{12}}\left( \left. \begin{matrix}
		-{{\gamma }_{1}},1,{{-\gamma }_{1}};\,\,\,\,\,\,\,\,\,\,\,\,\,\,\,\,\,\,\,\,\,\,\,\,\,\,\,  \\
		-{{\beta }_{1}},\alpha ,{1-{\beta }_{1}};-{{\gamma }_{1}},{{-\gamma }_{1}};1,1;1,1  \\
	\end{matrix} \right|\begin{matrix}
		-\left| x-\xi+2an \right|{{t}^{-{{\beta }_{1}}}}  \\
		\delta {{t}^{\alpha }}  \\
	\end{matrix} \right) \right.}-$$
$$\left. -{{E}_{12}}\left( \left. \begin{matrix}
	-{{\gamma }_{1}},1,{{-\gamma }_{1}};\,\,\,\,\,\,\,\,\,\,\,\,\,\,\,\,\,\,\,\,\,\,\,\,\,\,\,  \\
	-{{\beta }_{1}},\alpha ,{1-{\beta }_{1}};-{{\gamma }_{1}},{{-\gamma }_{1}};1,1;1,1  \\
\end{matrix} \right|\begin{matrix}
	-\left| x+\xi+2an \right|{{t}^{-{{\beta }_{1}}}}  \\
	\delta {{t}^{\alpha }}  \\
\end{matrix} \right) \right].$$
Next let us compute the derivative of the function $\widetilde G\left( t,x,0,\xi \right)$ with respect to $t:$
$$\widetilde G_t\left( t,x,0,\xi \right)=\frac{\partial}{\partial t}\left(\frac{{{ t }^{{{-\beta }_{1}}}}}{2}\sum\limits_{n=-\infty }^{\infty }{\left[ {{E}_{12}}\left( \left. \begin{matrix}
		-{{\gamma }_{1}},1,{{-\gamma }_{1}};\,\,\,\,\,\,\,\,\,\,\,\,\,\,\,\,\,\,\,\,\,\,\,\,\,\,\,  \\
		-{{\beta }_{1}},\alpha ,{1-{\beta }_{1}};-{{\gamma }_{1}},{{-\gamma }_{1}};1,1;1,1  \\
	\end{matrix} \right|\begin{matrix}
		-\left| x-\xi+2an \right|{{t}^{-{{\beta }_{1}}}}  \\
		\delta {{t}^{\alpha }}  \\
	\end{matrix} \right) \right.}-\right.$$
$$\left.\left. -{{E}_{12}}\left( \left. \begin{matrix}
	-{{\gamma }_{1}},1,{{-\gamma }_{1}};\,\,\,\,\,\,\,\,\,\,\,\,\,\,\,\,\,\,\,\,\,\,\,\,\,\,\,  \\
	-{{\beta }_{1}},\alpha ,{1-{\beta }_{1}};-{{\gamma }_{1}},{{-\gamma }_{1}};1,1;1,1  \\
\end{matrix} \right|\begin{matrix}
	-\left| x+\xi+2an \right|{{t}^{-{{\beta }_{1}}}}  \\
	\delta {{t}^{\alpha }}  \\
\end{matrix} \right) \right]\right)=$$
$$=\frac{\partial}{\partial t}\left( \sum\limits_{n=-\infty }^{+\infty }{\sum\limits_{k=0}^{+\infty }{\frac{{{\left( -1 \right)}^{k}}\left[ {{\left| x-\xi+2na \right|}^{k}}-{{\left| x+\xi+2na \right|}^{k}} \right]}{2\Gamma(k+1)}}}\sum_{i=0}^{+\infty}\frac{({-{\gamma }_{1}}-{{\gamma }_{1}}k)_i\delta^it^{\alpha i-{{\beta }_{1}}-\beta_1k}}{i!\,\Gamma(\alpha i+1-\beta_1-\beta_1k)}\right)=$$
$$= \sum\limits_{n=-\infty }^{+\infty }{\sum\limits_{k=0}^{+\infty }{\frac{{{\left( -1 \right)}^{k}}\left[ {{\left| x-\xi+2na \right|}^{k}}-{{\left| x+\xi+2na \right|}^{k}} \right]}{2\Gamma(k+1)}}}\sum_{i=0}^{+\infty}\frac{({-{\gamma }_{1}}-{{\gamma }_{1}}k)_i\delta^it^{\alpha i-{{\beta }_{1}}-\beta_1k-1}}{i!\,\Gamma(\alpha i-\beta_1-\beta_1k)}.$$
Now let us evaluate the following integral:
$$\int\limits_{0}^{t}{(t-s)^{-\beta} E_{\alpha, 1-\beta}^{-\gamma}\left[\mathcal{\delta}(t-s)^\alpha\right]\widetilde G_s\left( s,x,0,\xi \right)ds}=\int\limits_{0}^{t}{(t-s)^{-\beta} E_{\alpha, 1-\beta}^{-\gamma}\left[\mathcal{\delta}(t-s)^\alpha\right]\times}$$
$$\times\sum\limits_{n=-\infty }^{+\infty }{\sum\limits_{k=0}^{+\infty }{\frac{{{\left( -1 \right)}^{k}}\left[ {{\left| x-\xi+2na \right|}^{k}}-{{\left| x+\xi+2na \right|}^{k}} \right]}{2\Gamma(k+1)}}}s^{-{{\beta }_{1}}-\beta_1k-1}\sum_{i=0}^{+\infty}\frac{({-{\gamma }_{1}}-{{\gamma }_{1}}k)_i\delta^is^{\alpha i}}{i!\,\Gamma(\alpha i-\beta_1-\beta_1k)}ds=$$
$$=\int\limits_{0}^{t}{(t-s)^{-\beta} \sum\limits_{n=-\infty }^{+\infty }{\sum\limits_{k=0}^{+\infty }{\frac{{{\left( -1 \right)}^{k}}\left[ {{\left| x-\xi+2na \right|}^{k}}-{{\left| x+\xi+2na \right|}^{k}} \right]}{2\Gamma(k+1)}}}s^{-{{\beta }_{1}}-\beta_1k-1}\times}$$
$$\times \sum_{j=0}^{+\infty}\frac{(-\gamma)_j\delta^j(t-s)^{\alpha j}}{j!\Gamma(\alpha j+1-\beta)}\sum_{i=0}^{+\infty}\frac{({-{\gamma }_{1}}-{{\gamma }_{1}}k)_i\delta^is^{\alpha i}}{i!\,\Gamma(\alpha i-\beta_1-\beta_1k)}ds.$$
We apply the formula \eqref{eq2.21} to the sums with respect to $i$ and $j$ and perform the substitution $s=tz$:
$$\sum\limits_{n=-\infty }^{+\infty }{\sum\limits_{k=0}^{+\infty }{\frac{{{\left( -1 \right)}^{k}}\left[ {{\left| x-\xi+2na \right|}^{k}}-{{\left| x+\xi+2na \right|}^{k}} \right]}{2\Gamma(k+1)}}}\times$$
$$\times \sum_{j=0}^{+\infty}\sum_{i=0}^{j}\frac{(-\gamma)_i({-{\gamma }_{1}}-{{\gamma }_{1}}k)_{j-i}\delta^j}{i!(j-i)!\Gamma(\alpha i+1-\beta)\Gamma(\alpha j -\alpha i-\beta_1-\beta_1k)}\int\limits_{0}^{t}{(t-s)^{\alpha i-\beta} s^{\alpha j-\alpha i-{{\beta }_{1}}-\beta_1k-1}ds}=$$
$$=\sum\limits_{n=-\infty }^{+\infty }{\sum\limits_{k=0}^{+\infty }{\frac{{{\left( -1 \right)}^{k}}\left[ {{\left| x-\xi+2na \right|}^{k}}-{{\left| x+\xi+2na \right|}^{k}} \right]}{2\Gamma(k+1)}}} \sum_{j=0}^{+\infty}\sum_{i=0}^{j}\frac{(-\gamma)_j({-{\gamma }_{1}}-{{\gamma }_{1}}k)_{j-i}\delta^jt^{\alpha j-\beta-\beta_1-\beta_{1}k}}{i!(j-i)!\Gamma(\alpha j -\beta_1-\beta_1k-\beta+1)}.$$
We use \eqref{eq2.22} and express the obtained result in terms of the generalized Mittag–Leffler function:
$$\sum\limits_{n=-\infty }^{+\infty }{\sum\limits_{k=0}^{+\infty }{\frac{{{\left( -1 \right)}^{k}}\left[ {{\left| x-\xi+2na \right|}^{k}}-{{\left| x+\xi+2na \right|}^{k}} \right]}{2\Gamma(k+1)}}}t^{-\beta-\beta_1-\beta_{1}k}\sum_{j=0}^{+\infty}\frac{({-\gamma-{\gamma }_{1}}-{{\gamma }_{1}}k)_{j}\delta^j t^{\alpha j}}{j!\Gamma(\alpha j -\beta_1-\beta_1k-\beta+1)}=$$
$$=\sum\limits_{n=-\infty }^{+\infty }{\sum\limits_{k=0}^{+\infty }{\frac{{{\left( -1 \right)}^{k}}\left[ {{\left| x-\xi+2na \right|}^{k}}-{{\left| x+\xi+2na \right|}^{k}} \right]}{2\Gamma(k+1)}}}t^{-\beta-\beta_1-\beta_{1}k}E_{\alpha, {1-\beta_1k-\beta_1-\beta}}^{-\gamma-\gamma _{1}-{{\gamma }_{1}}k}\left[\delta t^\alpha\right].$$

\subsection{Appendix A4.}

Let us evaluate $\widetilde G_{xx}$ as follows:
$$\frac{\partial^2}{\partial x^2}\widetilde G\left( t,x,0,\xi \right)=$$
$$=\frac{\partial^2}{\partial x^2}\left(\frac{{{ t }^{{{-\beta }_{1}}}}}{2}\sum\limits_{n=-\infty }^{\infty }{\left[ {{E}_{12}}\left( \left. \begin{matrix}
		-{{\gamma }_{1}},1,{{-\gamma }_{1}};\,\,\,\,\,\,\,\,\,\,\,\,\,\,\,\,\,\,\,\,\,\,\,\,\,\,\,  \\
		-{{\beta }_{1}},\alpha ,{1-{\beta }_{1}};-{{\gamma }_{1}},{{-\gamma }_{1}};1,1;1,1  \\
	\end{matrix} \right|\begin{matrix}
		-\left| x-\xi+2an \right|{{t}^{-{{\beta }_{1}}}}  \\
		\delta {{t}^{\alpha }}  \\
	\end{matrix} \right) \right.}-\right.$$
$$\left.\left. -{{E}_{12}}\left( \left. \begin{matrix}
	-{{\gamma }_{1}},1,{{-\gamma }_{1}};\,\,\,\,\,\,\,\,\,\,\,\,\,\,\,\,\,\,\,\,\,\,\,\,\,\,\,  \\
	-{{\beta }_{1}},\alpha ,{1-{\beta }_{1}};-{{\gamma }_{1}},{{-\gamma }_{1}};1,1;1,1  \\
\end{matrix} \right|\begin{matrix}
	-\left| x+\xi+2an \right|{{t}^{-{{\beta }_{1}}}}  \\
	\delta {{t}^{\alpha }}  \\
\end{matrix} \right) \right]\right)=$$
$$=\frac{\partial^2}{\partial x^2}\left( \sum\limits_{n=-\infty }^{+\infty }{\sum\limits_{k=0}^{+\infty }{\frac{{{\left( -1 \right)}^{k}}\left[ {{\left| x-\xi+2na \right|}^{k}}-{{\left| x+\xi+2na \right|}^{k}} \right]}{2\Gamma(k+1)}}}t^{-{{\beta }_{1}}-\beta_1k}E_{\alpha, 1-\beta_1-\beta_1k}^{-{\gamma }_{1}-{{\gamma }_{1}}k}\left[\delta t^\alpha\right]\right)=$$
$$=\frac{\partial^2}{\partial x^2}\left( \sum\limits_{n=-\infty }^{-1 }{\sum\limits_{k=0}^{+\infty }{\frac{\left[ {{( x-\xi+2na)}^{k}}-{{(x+\xi+2na)}^{k}} \right]}{2\Gamma(k+1)}}}t^{-{{\beta }_{1}}-\beta_1k}E_{\alpha, 1-\beta_1-\beta_1k}^{-{\gamma }_{1}-{{\gamma }_{1}}k}\left[\delta t^\alpha\right]\right)+$$
$$+\frac{\partial^2}{\partial x^2}\left( {\sum\limits_{k=0}^{+\infty }{\frac{{(-1)^k( x-\xi)}^{k}}{2\Gamma(k+1)}}}t^{-{{\beta }_{1}}-\beta_1k}E_{\alpha, 1-\beta_1-\beta_1k}^{-{\gamma }_{1}-{{\gamma }_{1}}k}\left[\delta t^\alpha\right]+\right.{\sum\limits_{k=0}^{+\infty }{\frac{{(-1)^k( \xi-x)}^{k}}{2\Gamma(k+1)}}}t^{-{{\beta }_{1}}-\beta_1k}E_{\alpha, 1-\beta_1-\beta_1k}^{-{\gamma }_{1}-{{\gamma }_{1}}k}\left[\delta t^\alpha\right]-$$
$$-\left.{\sum\limits_{k=0}^{+\infty }{\frac{{(-1)^k( x+\xi)}^{k}}{2\Gamma(k+1)}}}t^{-{{\beta }_{1}}-\beta_1k}E_{\alpha, 1-\beta_1-\beta_1k}^{-{\gamma }_{1}-{{\gamma }_{1}}k}\left[\delta t^\alpha\right] \right)+$$
$$+\frac{\partial^2}{\partial x^2}\left( \sum\limits_{n=1 }^{+\infty }{\sum\limits_{k=0}^{+\infty }{\frac{(-1)^k\left[ {{( x-\xi+2na)}^{k}}-{{(x+\xi+2na)}^{k}} \right]}{2\Gamma(k+1)}}}t^{-{{\beta }_{1}}-\beta_1k}E_{\alpha, 1-\beta_1-\beta_1k}^{-{\gamma }_{1}-{{\gamma }_{1}}k}\left[\delta t^\alpha\right]\right)=$$
$$=\sum\limits_{n=-\infty }^{+\infty }{\sum\limits_{k=0}^{+\infty }{\frac{{{\left( -1 \right)}^{k}}\left[ {{\left| x-\xi+2na \right|}^{k}}-{{\left| x+\xi+2na \right|}^{k}} \right]}{2\Gamma(k+1)}}}t^{-\beta-\beta_1-\beta_{1}k}E_{\alpha, {1-\beta_1k-\beta_1-\beta}}^{-\gamma-\gamma _{1}-{{\gamma }_{1}}k}\left[\delta t^\alpha\right].$$

\bibliographystyle{plain}

\begin{thebibliography}{99}
\normalsize

\bibitem{Sarah}
S. Aljohani, M. Rashid, A. Kalsoom, N. Mlaiki.
Exploring the influence of generalized kernels on Green’s function in fractional differential equations.
{\it Int. J. Anal. Appl.} \textbf{22} (2024), Paper No. 188.


\bibitem{Al‑Refai}
M. Al-Refai, A. Nusseir, S. Al-Sharif.
Maximum principles for fractional differential inequalities with Prabhakar derivative and their applications.
{\it Fractal Fract.} \textbf{6} (2022), no. 10, Paper No. 612.

\bibitem{Asjad}
M.I. Asjad, M. Zahid, Y.-M. Chu, D. Baleanu.
Prabhakar fractional derivative and its applications in the transport phenomena containing nanoparticles.
{\it Thermal Science}, {\bf 2} (2021), pp.411-416.

\bibitem{El-Sayed}
El-Sayed, A.A., Boulaaras, S., Al-Kharousi, F.A.:
Dickson polynomial-based solutions for fractional order physics problems.
{J. Inequal. Appl.}, 118:118 (2025).
DOI: 10.1186/s13660-025-03367-7

\bibitem{Fernandez}
Fernandez, A., Restrepo, J.E., Suragan, D.:
Prabhakar-type linear differential equations with variable coefficients.
{Differ. Integral Equ.} \textbf{35}(9--10), 581--610 (2022).
DOI: 10.57262/die035-0910-581

\bibitem{Garra}
Garra, R., Gorenflo, R., Polito, F., Tomovski, Z.:
Hilfer–Prabhakar derivatives and some applications.
{Appl. Math. Comput.} \textbf{242}, 576--589 (2014).
DOI: 10.1016/j.amc.2014.05.129

\bibitem{Giusti}
Giusti, A., Colombaro, I., Garra, R., et al.:
A practical guide to Prabhakar fractional calculus.
{Fract. Calc. Appl. Anal.} \textbf{23}(1), 9--54 (2020).
DOI: 10.1515/fca-2020-0034

\bibitem{Tom2} Gorska, K., Pietrzak, T., Sandev, T.,  Tomovski, Z.: Volterra-Prabhakar derivative of distributed order and some applications. Journal of Computational and Applied Mathematics, 433. (2023)

\bibitem{Isah} Isah, S. S., Fernandez, A.,  Özarslan, M. A.: On bivariate fractional calculus with general univariate analytic kernels. Chaos, Solitons \& Fractals, 171 (2023).

\bibitem{Karimov}
Karimov, E., Hasanov, A.:
On a boundary-value problem in a bounded domain for a time-fractional diffusion equation with the Prabhakar fractional derivative.
{Bull. Karaganda Univ. Math. Ser.} \textbf{111}(3), 39--46 (2023).
DOI: 10.31489/2023m3/39-46

\bibitem{Karimov2}
Karimov, E., Kerbal, S., Turdiev, K.:
Direct and inverse problems with a dynamical condition for the sub-diffusion equation involving the Hilfer-Prabhakar integral-differential operator.
Rendiconti del Circolo Matematico di Palermo, II. Ser., 75:35 (2026).
DOI: 10.1007/s12215-025-01351-0

\bibitem{Karimov3} 
Karimov~E., Usmonov~D., Mirzaeva~M.;
Green’s Function and Solution Representation for a Boundary Value Problem Involving the Prabhakar Fractional Derivative.
Preprint, 	arXiv:2512.21259 [math.AP], 2026.

\bibitem{Kerbal}
Kerbal, S., Khasanov, Sh.:
On Katugampola-Prabhakar Fractional Integral-Differential Operators.
Gulf Journal of Mathematics, 20(1), 190–221 (2025).
DOI: https://doi.org/10.56947/gjom.v20i.2856

\bibitem{jonibek} Khujakulov, J. R.:  Initial-boundary value problem for a time fractional 
differential equation with the Prabhakar derivative on a star graph.  Bulletin of the 
Institute of Mathematics, 6(2), 20-30 (2023).

\bibitem{Kilbas}
Kilbas, A.A., Srivastava, H.M., Trujillo, J.J.:
Theory and Applications of Fractional Differential Equations.
Elsevier, Amsterdam (2006).

\bibitem{Knopp}
Knopp, K.:
Theory and application of infinite series.
Dover Publications, New York (1990).


\bibitem{Magar}
Magar, S.K., Dole, P.V., Ghadle, K.P.:
Prabhakar and Hilfer–Prabhakar fractional derivatives in the setting of $\Psi$-fractional calculus and its applications.
{Kragujevac J. Math.} \textbf{48}(4), 515--533 (2024).
DOI: 10.46793/KgJMat2404.515M

\bibitem{Mamanazarov}
Mamanazarov, A., Khalilov, K., Kodiraliev, A.:
Direct and inverse source problems with non-local boundary conditions for a time-fractional time and space degenerate heat equation.
Rendiconti del Circolo Matematico di Palermo Series 2, 74:162 (2025).
DOI: 10.1007/s12215-025-01286-6

\bibitem{Odibat}
Odibat, Z., Momani, S.M.:
Fractional Green's function for fractional partial differential equations.
{J. Eur. Syst. Autom.} \textbf{42}(6--8), 639--651 (2008).
DOI: 10.3166/jesa.42.639-651

\bibitem{Podlubny}
Podlubny, I.:
Fractional Differential Equations.
Academic Press, San Diego (1999).

\bibitem{Tom1} Polito, F.,  Tomovski, Z.: (2016). Some properties of Prabhakar-type fractional calculus operators. Fractional Differential Calculus, 6(1), 73--94 (2016).

\bibitem{Prabhakar}
Prabhakar, T.R.:
A singular integral equation with a generalized Mittag-Leffler function in the kernel.
{Yokohama Math. J.} \textbf{19}, 7--15 (1971).

\bibitem{Prudnikov}
Prudnikov, A.P., Brychkov, Yu.A., Marichev, O.I.:
Integrals and Series: Special Functions.
Nauka, Moscow (1983).

\bibitem{Pskhu}
Pskhu, A.V.:
Green functions of the first boundary-value problem for a fractional diffusion–wave equation in multidimensional domains.
{Mathematics} \textbf{8}(4), 464 (2020).
DOI: 10.3390/math8040464

\bibitem{Pskhu 2}
Pskhu, A.V.:
Uravneniya v chastnykh proizvodnykh drobnogo poryadka.
Nauka, Moscow (2005).

\bibitem{Noosha} Rani, N.,  Fernandez, A.: Mikusiński's operational calculus for Prabhakar fractional calculus. Integral Transforms and Special Functions, {\bf 33}, 12, 945--965 (2022).

\bibitem{Suzuki}
Suzuki, J., Zayernouri, M., D’Elia, M.:
A survey of fractional-order models in transport and anomalous materials (Technical Report).
DOE Tech. Rep. SAND2021-11291R (2021).
DOI: 10.2172/1820001

\bibitem{Turdiev 1}
Turdiev, Kh.:
Nonlocal problem for a diffusion–wave equation involving regularized Prabhakar fractional derivative.
{Bull. Inst. Math.} \textbf{6}(6), 39--45 (2023).

\bibitem{Turdiev 2}
Karimov, E., Turdiev, Kh., \& Usmonov, D.: Fractional generalization of hyperbolic-type equation and bivariate, trivariate Mittag-Leffler type functions. {Arab. J. Math.,} 1-18 (2026). 
\href{https://doi.org/10.1007/s40065-025-00603-2}{https://doi.org/10.1007/s40065-025-00603-2}9

\bibitem{Usmonov}
Usmonov, D., Mirzaeva, M.:
A Cauchy problem for the sub-diffusion equation with the Prabhakar fractional derivative.
{Gulf J. Math.} \textbf{21}(2), 181--203 (2025).
DOI: 10.56947/gjom.v21i2.3708

\bibitem{WMT} Wang, W.,  Metzler, R. Tomovski, Ž.:
Distributed-order fractional diffusion equation with Hilfer–Prabhakar fractional derivative.
Physica A: Statistical Mechanics and its Applications, {\bf 689}, 131370 (2026).

\end{thebibliography}

\end{document}